\documentclass[11pt,a4paper,reqno]{amsart}
\usepackage[includehead,includefoot,margin=27mm]{geometry}
\usepackage{times}
\usepackage{amsfonts} %
\usepackage{amsmath} %
\usepackage{amssymb} %
\usepackage{amsthm} %
\usepackage{enumerate} %
\usepackage{bbm} %
\usepackage{graphicx} %
\usepackage{nicefrac} %
\usepackage{color} %
\usepackage{hyperref}
\hypersetup{
    colorlinks=true,       
    linkcolor=blue,          
    citecolor=blue,        
    filecolor=blue,      
    urlcolor=blue           
}

\newtheorem{theorem}{Theorem}[] %
\newtheorem{lemma}[theorem]{Lemma} %
\newtheorem{proposition}[theorem]{Proposition} %
{\theoremstyle{remark} %
  \newtheorem{remark}[theorem]{Remark}} %
{\theoremstyle{definition} %
}

\newcommand{\CC}[0]{\ensuremath{\mathbb{C}}}
\newcommand{\RR}[0]{\ensuremath{\mathbb{R}}}
\newcommand{\QQ}[0]{\ensuremath{\mathbb{Q}}}
\newcommand{\ZZ}[0]{\ensuremath{\mathbb{Z}}}

\newcommand{\KK}[0]{\ensuremath{\mathrm{k}}}
\newcommand{\TT}[0]{\ensuremath{\mathrm{T}}}
\newcommand{\GM}[0]{\ensuremath{\mathbb{G}_{\mathrm{m}}}}
\newcommand{\GA}[0]{\ensuremath{\mathbb{G}_{\mathrm{a}}}}
\newcommand{\RT}[0]{\ensuremath{\mathcal{R}}}
\newcommand{\OO}[0]{\ensuremath{\mathcal{O}}}

\newcommand{\id}[0]{\ensuremath{\operatorname{id}}}

\newcommand{\aut}[0]{\ensuremath{\operatorname{Aut}}}
\newcommand{\homo}[0]{\ensuremath{\operatorname{Hom}}}
\newcommand{\spec}[0]{\ensuremath{\operatorname{Spec}}}
\newcommand{\lie}[0]{\ensuremath{\operatorname{Lie}}}
\newcommand{\Span}[0]{\ensuremath{\operatorname{span}}}

\usepackage{marvosym}

\title{Automorphisms of products of toric varieties}

\author{Alvaro Liendo} %
\address{Instituto de Matem\'atica y F\'isica, Universidad de Talca,
  Casilla 721, Talca, Chile} %
\email{aliendo@inst-mat.utalca.cl}
\thanks{{\it 2020 Mathematics Subject
    Classification}: 14M25.\\
  \mbox{\hspace{11pt}}{\it Key words}: complete toric varieties, automorphism groups.}

\author{Giancarlo Lucchini Arteche}
\address{Departamento de Matem\'aticas, Facultad de Ciencias, Universidad de Chile, Las Palmeras 3425, \~Nu\~noa, Santiago, Chile.}
\email{luco@uchile.cl}

\date{}

\begin{document}

\begin{abstract}
We give an explicit description of the automorphism group of a product of complete toric varieties over an arbitrary field in terms of the respective automorphism groups of its components. More precisely, we prove that, up to permutation of isomorphic components, an automorphism of a product corresponds to a product of automorphisms of its components. We also reprove, in modern language, the classic result by Demazure describing the group-scheme of automorphisms of a complete toric variety over an arbitrary field.
\end{abstract}

\maketitle

\section*{Introduction}

Let $\KK$ be a field. A \emph{variety} is a geometrically integral scheme of finite type over $\KK$. A (split) algebraic torus $\TT$ is an algebraic group isomorphic to $\GM^n$ for some integer $n$. A \emph{toric variety} $X$ is a normal variety having $\TT$ as an open set such that the action of $\TT$ on itself by translations extends to an action on $X$. We say that a toric variety $X$ is \emph{decomposable} if it can be written as a product $X=X_1\times X_2$ of two toric varieties of strictly smaller dimension. Otherwise we say that $X$ is \emph{indecomposable}.

In this note, we consider complete (i.e.~proper) toric varieties. If $X$ is such a variety, then its sheaf of automorphisms is represented by a locally algebraic group denoted by $\aut_X$ (cf.~\cite[Thm.~3.7]{MO}). Define $\aut^0_X$ as the neutral component of $\aut_X$ (it is thus an algebraic group). The aim of this short note is to prove the following theorem.

\begin{theorem}\label{main thm}
Let $X_1,\ldots,X_n$ be pairwise non-isomorphic indecomposable complete toric varieties.
For $1\leq i\leq n$, let $r_i$ be a positive integer. Let $X=\prod_{i=1}^n X_{i}^{r_i}$. Then
\[\aut_X\simeq \prod_{i=1}^n(\aut_{X_i}^{r_i}\rtimes S_{r_i}),\]
where $S_{r_i}$ is seen as a finite constant group-scheme acting on $\aut_{X_i}^{r_i}$ by permuting coordinates.
\end{theorem}

The main ingredient in the proof of Theorem \ref{main thm} is Demazure's description of $\aut_X$ for a complete toric variety $X$, which is given in the foundational paper \cite{De}. Unfortunately, this description is stated for a \emph{smooth} complete toric variety $X$, even though the proofs (which are spread all over the article) naturally generalize to \emph{any} complete toric variety. Following Demazure's ideas, we give a short proof in modern language of such a description for an \emph{arbitrary} complete toric variety over any base field $\KK$ (for notations and definitions, see Section \ref{sec:toric}).

\begin{theorem}\label{thm Demazure}
Let $\Sigma$ be a complete fan and let $X$ be the corresponding complete toric variety. Let $\RT(\Sigma)$ be the set of Demazure roots of $\Sigma$. For $e\in\RT(\Sigma)$, let $\rho_e$ be the corresponding character and define the rational map
\begin{align} \label{eq:root-subgroup}
  \alpha_e\colon \GA\times \TT\dashrightarrow \TT,\quad \mbox{given by}\quad (s,t)\mapsto t\cdot \lambda_{\rho_e}(1+s\chi^e(t))\,.
\end{align}
\begin{enumerate}[$(i)$]
  \item The rational map $\alpha_e$ defines a faithful $\GA$-action on $X$. In particular, the morphism $\iota_e:\GA\to\aut_X$ induced by $\alpha_e$ is a closed immersion.\label{Dem1}
  \item The neutral component $\aut^0_X$ is spanned by $\TT$ and the image of $\iota_e$ for all $e\in \RT(\Sigma)$. In particular, $\aut_X$ is smooth.\label{Dem2}
  \item The quotient group-scheme $\aut_X/\aut^0_X$ is a finite constant group-scheme generated by the images of toric automorphisms.\label{Dem3}
  \end{enumerate}  
\end{theorem}

Alternative proofs of this description exist in the literature for the automorphism \emph{group} (not the \emph{scheme}) of a complete toric variety. One of these is given in \cite{CoxJAG} for complete $\QQ$-factorial toric varieties over $\CC$. A similar result was later obtained in \cite{BG} for projective toric varieties over any base field. There is also the recent preprint \cite{MSS}, which works with arbitrary complete toric varieties over an algebraically closed field of characteristic $0$. The main feature of Theorem \ref{thm Demazure} is that it describes $\aut_X$ as a scheme, proving in particular that it is \emph{smooth}. We should remark however that we work over a base field $\KK$, while Demazure's results are valid over $\ZZ$, so there is still room for improvement.\\

Theorem \ref{thm Demazure} has consequences in the theory of twisted forms of toric varieties, which was developed by Duncan in \cite{Duncan1}. Indeed, when studying automorphisms in the proper split case (\S4.2 in \textit{loc.~cit.}), Duncan is led to assume that $\aut_X$ is a smooth algebraic group. Theorem \ref{thm Demazure} proves that this assumption is unnecessary. Moreover, Theorem \ref{thm Demazure} naturally yields an analogous description of the automorphism group of \emph{any} $\KK$-form of a toric variety: if such a $\KK$-form is defined by twisting a (split) toric variety $X$ with a cocycle $\alpha$ representing a class in $H^1(\KK,\aut_X)$, then its automorphism group is simply the inner twist ${}_\alpha\aut_X$ of the group-scheme $\aut_X$.

\subsection*{Acknowledgements}
We would like to thank Michel Brion for his comments on preliminary versions of this note and for encouraging us to prove Theorem \ref{thm Demazure}. We also thank Alexander Duncan for helpful email discussions and an anonymous referee for her or his suggestions concerning the layout of this note. The first author was partially supported by ANID via Fondecyt Grant 1200502. The second author was partially supported by ANID via Fondecyt Grant 1210010.

\section{Preliminaries on toric varieties}
  \label{sec:toric}

Let $M$ be a lattice, i.e., a finitely generated free abelian group and let $N=\homo(M,\ZZ)$ be its dual lattice. We let $M_{\RR}=M\otimes_\ZZ\RR$, $N_{\RR}=N\otimes_\ZZ\RR$, and $\langle\cdot,\cdot\rangle:N_\RR\times M_\RR\rightarrow \RR$ be the duality pairing given by $\langle p,m\rangle=p(m)$. Let $\TT=N \otimes_\ZZ \GM$ be the algebraic torus whose character lattice is $M$ and whose 1-parameter subgroup lattice is $N$. As customary, for every $m\in M$ we denote by $\chi^m\colon \TT\rightarrow \GM$ the corresponding character and for every $p\in N$ we denote by $\lambda_p\colon \GM\rightarrow \TT$ the corresponding 1-parameter subgroup.
  
 Recall that a toric variety $X$ is a normal variety having $\TT$ as an open set such that the action of $\TT$ on $\TT$ by translations extends to an action on $X$. Let $X$ and $X'$ be toric varieties with corresponding tori $\TT$ and $\TT'$, respectively. A toric morphism is a morphism $\varphi\colon X\rightarrow X'$ that restricts to a morphism of algebraic groups $\TT\rightarrow \TT'$.

 A strictly convex polyhedral cone $\sigma$ in $N_\RR$ is called rational if $\sigma$ is spanned by a finite set of vectors in the lattice $N$. All cones will be strictly convex, rational and polyhedral. We will refer to them simply as cones. A fan $\Sigma\in N_\RR$ is a finite collection of cones such that every face of a cone $\sigma\in\Sigma$ is contained in the fan $\Sigma$; and for all cones $\sigma,\sigma'\in\Sigma$ the intersection $\sigma\cap\sigma'$ is a face in both cones $\sigma$ and $\sigma'$. Let $\Sigma$ and $\Sigma'$ be fans in $N_\RR$ and $N'_\RR$, respectively. A morphism of fans is a linear map $\psi\colon N_\RR\rightarrow N'_\RR$ restricting to a homomorphism $N\rightarrow N'$ such that for every $\sigma\in \Sigma$, there exists $\sigma'\in \Sigma'$ with $\psi(\sigma)\subset \sigma'$.

Let $\Sigma$ be a fan in $N_\RR$. A toric variety $X_\Sigma$ is defined from $\Sigma$ as follows. For every $\sigma\in \Sigma$ we define the affine toric variety $X_\sigma=\spec \KK[\sigma^\vee\cap M]$ where \[\sigma^\vee:=\{m\in M_\RR\mid \langle p,m\rangle\geq 0\, \mbox{for all } p\in \sigma\}\in M_\RR,\]
is the dual cone of $\sigma\in N_\RR$ and $\KK[\sigma^\vee\cap M]$ is the semigroup algebra defined via
$$\KK[\sigma^\vee\cap M]=\bigoplus_{m\in \sigma^\vee\cap M} \KK\cdot\chi^m\quad\mbox{where}\quad \chi^m\chi^{m'}=\chi^{m+m'}\mbox{ and } \chi^0=1\,.$$
Furthermore, for every $\sigma,\sigma'\in \Sigma$ the algebra inclusions $\KK[\sigma^\vee\cap M]\hookrightarrow \KK[(\sigma\cap \sigma')^\vee\cap M]\hookleftarrow\KK[(\sigma')^\vee\cap M]$ induce $\TT$-equivariant affine open embeddings $X_\sigma\hookleftarrow X_{\sigma\cap\sigma'}\hookrightarrow X_{\sigma'}$. The toric variety $X_\Sigma$ is obtained by glueing the family of affine toric varieties $\{X_\sigma\mid \sigma\in \Sigma\}$ along these embeddings.

The construction described above can also be defined for morphisms. This yields an equivalence of categories between the category of toric varieties with toric morphisms and the category of fans with morphisms of fans. In particular, we recall that the toric variety $X_\Sigma$ is complete if and only if the union of the cones of $\Sigma$ equals $N_\RR$. We refer to any textbook on toric geometry such as \cite{Cox,F93,T88} for details on this equivalence.

We end this section by recalling the definition of a Demazure root, so that every object in Theorem \ref{thm Demazure} is correctly defined. Let $\Sigma$ be a fan in $N_\RR$. For every non-negative integer $i$, we denote by $\Sigma(i)$ the $i$-skeleton of $\Sigma$, i.e., the set of $i$-dimensional cones in $\Sigma$. As usual, we identify a ray $\rho\in \Sigma(1)$ with its primitive vector in $N$. A Demazure root of the fan $\Sigma$ is a lattice vector $e\in M$ such that there exists $\rho_e \in \Sigma(1)$ with $\langle \rho_e,e\rangle=-1$ and $\langle \rho,e\rangle\geq 0$ for all $\rho\in \Sigma(1)\setminus \{\rho_e\}$. Note that such a $\rho_e$ is unique. We denote the set of roots of $\Sigma$ by $\RT(\Sigma)$.

\section{Proof of Theorem \ref{thm Demazure}}
\label{sec:proof-demazure}

In this section we prove Theorem \ref{thm Demazure} following \cite{De}. We would like to thank Michel Brion for his help with some parts of this proof.

\subsection*{Proof of Theorem~\ref{thm Demazure}~\eqref{Dem1}}
Following \cite[Sec.~4, \S5, Thm.~3]{De}, we compute first the comorphism corresponding to \eqref{eq:root-subgroup} at the level of function fields. Note that $\alpha_e$ can be decomposed as the composition
\[\GA\times\TT \stackrel{\mu_1}{\dashrightarrow} \GM\times\GA\times\TT \stackrel{\mu_2}{\dashrightarrow} \GA\times\TT \stackrel{\mu_3}{\dashrightarrow} \TT,\]
where $\mu_1$ sends $(s,t)$ to $(\chi^e(t),s,t)$, $\mu_2$ sends $(r,s,t)$ to $(rs+1,t)$ and $\mu_3$ sends $(s,t)$ to $t\lambda_{\rho_e}(s)$. Each of these comorphisms is easily computable:
\begin{itemize}
    \item $\mu_3^*$ sends a character $\chi^m\in\OO_X(\TT)=\KK[M]$ to $s^{\langle\rho_e,m\rangle}\chi^m\in \KK(s)[M]$;
    \item $\mu_2^*$ sends $f\chi^m\in\KK(s)[M]$ to $f'\chi^m\in\KK(r,s)[M]$, where $f'\in\KK(r,s)$ is obtained from $f$ by replacing every appearance of $s$ by $rs+1$.
    \item $\mu_1^*$ sends $g\chi^m\in\KK(r,s)[M]$ to $g'\chi^m\in\mathrm{Frac}(\KK(s)[M])$, where $g'$ is obtained from $g$ by replacing every appearance of $r$ by $\chi^e$.
\end{itemize}
Putting all this together, we obtain
\begin{equation}\label{eqn Demazure}
\alpha_e^*(\chi^m)=\chi^m(1+s\chi^e)^{\langle\rho_e,m\rangle},\quad \mbox{for all } \chi^m\in\OO_X(\TT).
\end{equation}

Let $\sigma\in \Sigma$. We will show that $\alpha_e$ defines a morphism $\GA\times X_\sigma\to X$. Assume first that $\rho_e\in\sigma(1)$. In this case, $\alpha_e$ actually defines a morphism $\GA\times X_\sigma\rightarrow X_\sigma$. Indeed, letting $m\in \sigma^\vee\cap M$, we have that in \eqref{eqn Demazure} the character $\chi^{m+ie}$ appears with $0\leq i\leq \langle \rho_e,m\rangle$. Clearly $\langle\rho,m+ie\rangle\geq 0$ for $\rho\in\sigma(1)$ different from $\rho_e$. On the other hand, $\langle \rho_e,m+ie\rangle=\langle \rho_e,m\rangle-i\geq 0$. So $m+ie\in\sigma^\vee\cap M$ and the claim follows.

Assume now that $\rho_e\notin\sigma(1)$ and let $f=1+s\chi^e$ and $g=\chi^e$ in $\KK[\sigma^\vee\cap M][s]$. It is immediate that the standard open subvarieties $(\GA\times X_\sigma)_f$ and $(\GA\times X_\sigma)_g$ cover $\GA\times X_\sigma$, so it will suffice to prove that $\alpha_e$ induces morphisms $(\GA\times X_\sigma)_f\rightarrow X_{\sigma}$ and $(\GA\times X_\sigma)_g\rightarrow X_{\sigma'}$ for a suitable $\sigma'\in\Sigma$. This is evident for $f$ by \eqref{eqn Demazure}. As for $g$, let $\sigma'$ be the cone\footnote{Note that the fact that $\sigma'$ is a cone in $\Sigma$ is part of Demazure's definition of a root in \cite[Sec.~4, \S5, Def.~4]{De}. However, Demazure proves (cf.~\cite[Sec.~4, \S5, Rem.~3]{De}) that in the complete case this is in fact a consequence of our current definition of a Demazure root.} in $\Sigma$ spanned by $\rho_e$ and $\sigma\cap e^\bot$. Letting $m\in (\sigma')^\vee\cap M$ we have 
$$\alpha_e^*(\chi^m)=\chi^{m+\ell e}(1+s\chi^e)^{\langle \rho_e,m\rangle} \chi^{-\ell e}, \quad \mbox{for any integer }  \ell. $$
The last two factors in this product are already morphisms $(\GA\times X_\sigma)_g\to X_{\sigma'}$ since $g=\chi^e$ and $\rho_e\in\sigma'$. To conclude, it suffices to prove that there exists an integer $\ell$ such that
\[\langle m+\ell e, \rho\rangle=\langle m,\rho\rangle+\ell\langle e, \rho\rangle\geq 0, \quad \mbox{for all }\rho\in \sigma(1).\]
Now, if $\langle e, \rho\rangle=0$, then $\langle m,\rho\rangle\geq 0$ by definition of $\sigma'$ and we are done. Otherwise, $\langle e, \rho\rangle>0$ and it is enough to take $\ell \geq -\langle m,\rho\rangle$. Since there are finitely many rays in $\sigma$, we may always choose $\ell$ big enough.

We have thus proved that $\alpha_e$ defines a morphism $\GA\times X_\sigma\to X$ for any cone $\sigma\in\Sigma$. And since the $X_\sigma$ cover $X$, this proves that $\alpha_e$ is a morphism $\GA\times X\to X$.\\

We are left with the faithfulness of the action, which implies that the morphism $\iota_e:\GA\to\aut_X$ defined by $\alpha_e$ is a closed immersion. This is stated outright by Demazure in \cite[Sec.~4, \S5]{De} right after proving his Th\'eor\`eme 3, but it is slightly less evident in the setting of non-smooth varieties.

In order to check this, we need to prove that, for every $\KK$-algebra $A$ and for $s\in \GA(A)=A$, we have the following implication (we abusively omit the notation $\spec$ in what follows):
\[\iota_e(s)=\id_{X\times_\KK A} \Rightarrow s=0.\]

Let $s\in A$ be as above and fix a cone $\sigma\in\Sigma$ containing $\rho_e$. Note that, since $\iota_e(s)=\id_{X\times_\KK A}$, it sends $X_\sigma\times_\KK A$ to $X_\sigma\times_\KK A$. We may use then \eqref{eqn Demazure} in order to describe $\iota_e(s)$ restricted to $X_\sigma\times_\KK A$. Using this we see that $\chi^m\in A[\sigma^\vee\cap M]$ maps to $\chi^m(1+s\chi^e)^{\langle\rho_e,m\rangle}\in A[\sigma^\vee\cap M]$. Since $\rho_e$ is a primitive lattice vector and belongs to $\sigma$, there exists $m_0\in\sigma^\vee\cap M$ such that $\langle\rho_e,m_0\rangle=1$. Since by hypothesis $\iota_e(s)(\chi^{m_0})=\chi^{m_0}$, we get then $s\chi^{m_0+e}=0\in A[\sigma^\vee\cap M]$, which implies $s=0$, as desired. This concludes the proof of the faithfulness of the action defined by $\alpha_e$.

\subsection*{Proof of Theorem~\ref{thm Demazure}~\eqref{Dem2}}
This item corresponds to the first part of \cite[Sec.~4, \S6, Prop.~11]{De}, which we follow closely. For $e\in \RT(\Sigma)$ we have by part \eqref{Dem1} a closed immersion $\iota_e:\GA\to \aut_X$, while we also have the natural closed immersion $\iota_0:\TT\to\aut_X$. This gives injective morphisms $i_e:\lie(\GA)\to\lie(\aut_X)$ and $i_0:\lie(\TT)\to\lie(\aut_X)$, where we recall that $\lie(\aut_X)$ consists of derivations of $\OO_X$.

We claim\footnote{This claim corresponds to \cite[Sec.~4, \S5, Prop.~7]{De}, for which we give a simplified proof in the complete case.} that the images of $i_0$ and $i_e$ for $e\in\RT(\Sigma)$ generate $\lie(\aut_X)$. Since every derivation can be decomposed as a finite sum of homogeneous derivations, it is enough to show that every homogeneous derivation is in the image of $i_e$ for some $e\in\RT(\Sigma)\cup\{0\}$. Moreover, every derivation of $\OO_X$ induces a derivation of $\OO_X(\TT)=\KK[M]$ and it is easy to check that every homogeneous derivation of $\KK[M]$ is a multiple of
$$\partial_{p,e}\colon \KK[M]\to \KK[M],\quad \chi^m\mapsto \langle p,m\rangle\cdot \chi^{m+e},$$
for some $e\in M$ and some primitive $p\in N$. Assume now that $\partial_{p,e}$ is the restriction of a derivation of $\OO_X$. Then for every $\sigma\in\Sigma$ we have 
$\partial_{p,e}(\KK[\sigma^\vee\cap M])\subset \KK[\sigma^\vee\cap M]$, which is equivalent by the equation above to the following statement:
  \begin{align}
    \label{eq robada}
    \mbox{For every } m\in (\sigma^\vee\cap M)\setminus p^\bot, 
    \mbox{ we have }
    \langle\rho, m+e\rangle\geq 0\mbox{ for all } \rho\in
    \sigma(1)\,.
  \end{align}
If $e=0$, then $\partial_{p,e}$ corresponds to the infinitesimal generator of the one parameter subgroup $\lambda_p$ of $\TT$, hence it belongs to the image of $i_0$ and we are done. If $e\neq 0$, we claim that $e$ is a Demazure root, $p\in\{\pm\rho_e\}$ and $\partial_{p,e}$ lies in the image of $i_e$. Indeed, the completeness of $X$ implies that there exists $\rho_0\in \Sigma(1)$ such that $e\notin \rho_0^\vee$. Fix such a $\rho_0$ and assume that $p$ is not a multiple of $\rho_0$. Then there exists $m\in \rho_0^\vee\cap M$ such that $\langle \rho_0,m\rangle=0$ and $\langle p,m\rangle\neq 0$. Hence, \eqref{eq robada} implies that $\langle \rho_0,e\rangle\geq 0$ and so $e\in \rho_0^\vee$, which yields a contradiction. Thus, $p$ is a multiple of $\rho_0$, which proves in particular that $\rho_0$ is unique.

Finally, let $m\in \rho_0^\vee\cap M$ such that $\langle \rho_0, m\rangle=1$. Then \eqref{eq robada} implies that $\langle \rho_0,m+e\rangle \geq 0$, which yields $\langle \rho_0,e\rangle \geq -1$ and thus $\langle \rho_0,e\rangle =-1$. This proves that $e$ is a Demazure root and that $\rho_0=\rho_e$. Differentiating equation \eqref{eqn Demazure} we get that the infinitesimal generator of $\alpha_e$ is $\partial_{\rho_e,e}$, so that $\partial_{p,e}=\pm\partial_{\rho_e,e}$ lies in the image of $i_e$, proving the claim.

In order to conclude, by \cite[Exp.~VI\textsubscript{B}, Prop.~7.1]{SGA3}, we may consider the \emph{smooth} connected subgroup $G\subset\aut_X$ generated by the closed immersions $\iota_0$ and $\iota_e$ for $e\in \RT(\Sigma)$. The claim above implies that $\lie(G)=\lie(\aut_X)$, which immediately tells us that $\aut_X$ is smooth at the identity and hence everywhere. From this we get that $G=\aut^0_X$.

\subsection*{Proof of Theorem~\ref{thm Demazure}~\eqref{Dem3}}
This corresponds to the second part of \cite[Sec.~4, \S6, Prop.~11]{De}, whose proof is a direct application of \cite[Sec.~4, \S6, Prop.~10]{De}, which Demazure proves by ``reasoning as in the second corollary to \cite[Sec.~2, \S5, Thm.~4]{De}''. We give here a simplified proof that uses the well-known equivalence of categories between toric varieties and fans.

Note that there are finitely many toric automorphisms of $X$ and that they are all defined over $\KK$ since they come from automorphisms of the corresponding complete fan. We only have to prove then that these generate the quotient group-scheme $\aut_X/\aut_X^0$. And since $\aut_X$ is smooth by part \eqref{Dem2}, we only need to check this over an algebraically closed field $\KK$.

Let $\varphi\in\aut_X(\KK)$ be a representative of an element in the quotient. Since $\TT$, seen as automorphisms of $X$, is a maximal torus of $\aut_X$ and maximal tori are $\KK$-conjugate in $\aut_X^0$, up to composing $\varphi$ with a suitable element in $\aut_X^0$, we may assume that it normalizes $\TT$.

This assumption implies in turn that $\varphi$ stabilizes the torus $\TT$ inside $X$ since it is the unique open $\TT$-orbit in $X$. Then, Rosenlicht's Lemma (cf.~\cite[Lem.~6.5~(iii)]{Sansuc}) tells us that, up to composing again $\varphi$ with a suitable $\psi\in\TT\subset\aut_X^0$, we may assume that it restricts to a group automorphism of $\TT\subset X$, and this is the definition of a toric morphism. This proves that toric automorphisms generate the quotient, as desired.

\section{Proof of Theorem \ref{main thm}}
\label{sec:proof}

We start by considering the neutral component of the automorphism group, which amounts to proving the following result.

\begin{proposition}\label{prop aut0}
Let $X_1,\ldots,X_n$ be complete toric varieties and let $X=\prod_{i=1}^n X_{i}$. Then
\[\aut^0_X\simeq \prod_{i=1}^n\aut^0_{X_i}.\]
\end{proposition}

This proposition is actually already known for \emph{arbitrary} complete varieties as a consequence of Blanchard's Lemma (cf.~\cite[Cor.~7.2.3]{Brion}), but we give nevertheless a full proof here below that only uses Theorem \ref{thm Demazure}, providing a ``full-toric'' self-contained proof of Theorem \ref{main thm}.

\begin{remark}
Since Proposition \ref{prop aut0} is valid for arbitrary complete varieties, one is naturally led to ask whether this is also the case for Theorem \ref{main thm}. This is far from being true. Indeed, it suffices to consider a general elliptic curve $E$ and the product $E\times E$. Then $\aut_E\simeq E\rtimes\ZZ/2\ZZ$, while $\aut_{E\times E}\simeq (E\times E)\rtimes \mathrm{GL}_2(\ZZ)$. In particular, the quotient $\aut_{E\times E}/\aut^0_{E\times E}$ is infinite, while $\aut_E/\aut^0_E$ is finite.
\end{remark}

\begin{proof}[Proof of Proposition \ref{prop aut0}]

  It is enough to show the result for $n=2$, i.e., for $X\times X'$ with $X,X'$ complete toric varieties. We will show that $\aut^0_{X\times X'}$ is spanned by $\aut_X^0\times \{\id_{X'}\}$ and $\{\id_X\}\times \aut^0_{X'}$ proving the proposition.

  Let $\Sigma$ and $\Sigma'$ be the fans in $N_\RR$ and $N'_\RR$ corresponding to the toric varieties $X$ and $X'$, with acting tori $\TT$ and $\TT'$, respectively. By \cite[Prop.~3.1.14]{Cox} the fan of $X\times X'$ is
  \[\Sigma\times \Sigma'=\{\sigma\times\sigma'\mid \sigma\in \Sigma, \sigma'\in \Sigma'\}.\]
  Remark that for every Demazure root $e\in \RT(\Sigma)$ we obtain a root $(e,0)$ of $\Sigma\times \Sigma'$. Analogously, for every Demazure root $e'\in \RT(\Sigma')$ we obtain a root $(0,e')$ of $\Sigma\times \Sigma'$. Furthermore, by Theorem~\ref{thm Demazure}~\eqref{Dem2} we have that $\TT\times \{\id_{X'}\}$ and the images of $\iota_{(e,0)}$ span a group isomorphic to $\aut^0_X\times \{\id_{X'}\}$ inside $\aut^0_{X\times X'}$. Analogously, $\{\id_X\}\times \TT'$ and the images of $\iota_{(0,e')}$ span a group isomorphic to $\{\id_X\}\times \aut^0_{X'}$ inside $\aut^0_{X\times X'}$.

  Hence, to prove the proposition, it is enough by Theorem~\ref{thm Demazure}~\eqref{Dem2} to show that every Demazure root $\overline{e}$ of $\Sigma\times \Sigma'$ has the form $(e,0)$ for $e$ a Demazure root of $\Sigma$ or $(0,e')$ for $e'$ a Demazure root of $\Sigma'$. Let $\overline{e}=(e,e')$ be a Demazure root of $\Sigma\times \Sigma'$ and let $\rho_{\overline{e}}$ be the unique ray of $\Sigma\times \Sigma'$ such that $\langle \rho_{\overline{e}},\overline{e}\rangle=-1$. Since the 1-skeleton of $\Sigma\times\Sigma'$ is the disjoint union of $\Sigma(1)\times\{0\}$ and $\{0\}\times\Sigma'(1)$, we may assume, without loss of generality, that $\rho_{\overline{e}}=(\rho_0,0)$ for some $\rho_0\in \Sigma(1)$. Then for every ray $\rho'\in \Sigma'(1)$ we have $\langle (0,\rho'),(e,e')\rangle=\langle\rho',e'\rangle\geq 0$. Since $X'$ is complete, the support of $\Sigma'$ is $N'_\RR$ and hence $\langle\rho',e'\rangle\geq 0$ for every $\rho'\in \Sigma'(1)$ yields $e'=0$. Moreover, $\langle (\rho,0),(e,e')\rangle=\langle \rho,e\rangle$ and so $\langle \rho,e\rangle\geq 0$ for all $\rho\neq \rho_0$ and $\langle \rho_0,e\rangle=-1$. This shows that $e$ is a Demazure root of $\Sigma$, proving the proposition.
\end{proof}

The second ingredient in the proof of Theorem \ref{main thm} is the following lemma, which has the same flavour as a result by Ballard, Duncan and McFaddin on automorphisms of products of complete fans (cf.~\cite[Lem.~3.10]{Duncan2}).

\begin{lemma}\label{lemma injection}
Let $X,X_1',X_2'$ be complete toric varieties with $X$ indecomposable. Let $\Sigma,\Sigma_1',\Sigma_2'$ be the respective fans. Consider a toric morphism $\varphi:X\to X_1'\times X_2'$ and denote by $\psi$ the corresponding fan morphism. Assume that, for every $1\leq i\leq \dim(X)$ and for every $\sigma\in\Sigma(i)$, we have $\psi(\sigma)\in (\Sigma_1'\times\Sigma_2')(i)$. Then $\varphi(X)$ is contained in either $X_1'\times\{1\}$ or $\{1\}\times X_2'$.
\end{lemma}

\begin{proof}

Let $N$ and $M$ be the lattices corresponding to $X$ and let $N'_i$ and $M'_i$ be the lattices corresponding to $X_i'$ for $i=1,2$. In particular, $\Sigma$ is a fan in $N_{\RR}$ and $\Sigma_i'$ is a fan in $N'_{i,\RR}$. Note that the hypothesis on $\psi$, applied to a full dimensional cone in $\Sigma$, implies that $\psi:N_\RR\to N'_{1,\RR}\oplus N'_{2,\RR}$ is injective.

Since the 1-skeleton of $\Sigma_1'\times\Sigma_2'$ is the disjoint union of $\Sigma_1'(1)\times\{0\}$ and $\{0\}\times\Sigma_2'(1)$ and $\psi$ preserves 1-skeletons, we may take preimages and write $\Sigma(1)=B_1\sqcup B_2$. Define $N_{i,\RR}:=\Span(B_i)$ for $i=1,2$ and note that $N_\RR = N_{1,\RR}\oplus N_{2,\RR}$ since the images of $N_{1,\RR}$ and $N_{2,\RR}$ in $N'_{1,\RR}\oplus N'_{2,\RR}$ have trivial intersection and $X$ is complete. This decomposition of $N_\RR$ gives a natural decomposition $M_\RR=M_{1,\RR}\oplus M_{2,\RR}$, with $M_{i,\RR}$ canonically isomorphic to the dual of $N_{i,\RR}$.

Let $\sigma$ be a full dimensional cone in $\Sigma$ and write $\sigma(1)=S_1\sqcup S_2$ by taking intersections with $B_1$ and $B_2$ respectively and note that $N_{i,\RR}=\Span(S_i)$. Let $\sigma_i\subset N_{i,\RR}$ be the cone spanned by $S_i$. Let $u_i\in M_{i,\RR}$ be an element in the relative interior of $\sigma_i^\vee\in  M_{i,\RR}$, so that $\langle p,u_i\rangle>0$ for every $0\neq p\in\sigma_i$. Let  $H_i:=\{p\in N_\RR\mid \langle p,u_i\rangle=0\}$ be the supporting hyperplane in $N_\RR$ of the vector $u_i$. We have then $\sigma_1=H_2\cap \sigma$ and $\sigma_2=H_1\cap\sigma$, hence $\sigma_1$ and $\sigma_2$ are both faces of $\sigma$. In particular, $\sigma_1,\sigma_2\in\Sigma$.

Since $\psi$ preserves cones by hypothesis, we see that $\psi(\sigma),\psi(\sigma_1),\psi(\sigma_2)$ are cones in $\Sigma_1'\times\Sigma_2'$ and it is evident then that $\psi(\sigma)=\psi(\sigma_1)\times\psi(\sigma_2)$ and thus $\sigma=\sigma_1\times\sigma_2$. Since this is true for every full dimensional cone in $\Sigma$, this defines subfans $\Sigma_1,\Sigma_2\subset\Sigma$ such that $\Sigma$ can be seen as a subfan of $\Sigma_1\times\Sigma_2$, which has the same dimension. Since $X$ is complete, we see that $\Sigma=\Sigma_1\times\Sigma_2$, which contradicts the indecomposability of $X$ unless one of the $\Sigma_i$'s is trivial, which is equivalent to $B_i=\emptyset$. This implies the assertion in the statement of the Lemma.
\end{proof}

We are now ready to prove our main theorem.

\begin{proof}[Proof of Theorem \ref{main thm}]
Let $\varphi$ be an automorphism of $X$. Then by Proposition \ref{prop aut0} and Theorem~\ref{thm Demazure}~\eqref{Dem3}, we know that up to composing $\varphi$ with an element in
\[\prod_{i=1}^n(\aut^0_{X_i})^{r_i}=\aut^0_X,\]
we may assume that $\varphi$ is a toric automorphism, i.e.~it corresponds to an automorphism $\psi$ of the fan $\Sigma$ of $X$. Note that such an automorphism preserves $\Sigma(i)$ (in the sense of Lemma \ref{lemma injection}) for every $1\leq i\leq\dim X$.

Denote by $X_{i,j}$ with $1\leq j\leq r_i$ the different copies of $X_i$. Define then, for each $X_{i,j}$, the corresponding subvariety
\[Y_{i,j}:=\{1\}\times\cdots\times X_{i,j}\times\cdots\times\{1\}\subset X\]
Since these are indecomposable toric varieties, Lemma \ref{lemma injection} tells us that $\varphi(Y_{i,j})\subset Y_{i',j'}$ for some $i',j'$. And since $\varphi$ is invertible, we see that $Y_{i,j}\simeq Y_{i',j'}$, hence $X_{i,j}\simeq X_{i',j'}$ and thus $i=i'$. Then, up to composing $\varphi$ with a permutation of isomorphic components (which is a toric morphism), we may assume that $\varphi$ is toric and preserves each $Y_{i,j}$. This in turn implies that $\psi$ preserves each component of $\Sigma$, which is the direct product of fans of the $X_i$'s. And this gives $\varphi\in\prod_{i=1}^n\aut_{X_i}^{r_i}$, which concludes the proof of the theorem.
\end{proof}


\begin{thebibliography}{00}

\bibitem{Duncan2}
M.~Ballard, A.~Duncan, P.~McFaddin, \textit{On derived categories of arithmetic toric varieties}, Ann.~$K$-Theory 4(2), 2019, 211--242.

\bibitem{Brion}
M.~Brion, \textit{Some structure theorems for algebraic groups.} Algebraic groups: structure and actions, 53--126, Proc.~Sympos.~Pure Math.~94, Amer.~Math.~Soc., Providence, RI, 2017.

\bibitem{BG}
W.~Bruns, J.~Gubeladze, \textit{Polytopal Linear Groups}, J.~Algebra 218, 1999, 715--737.

\bibitem{CoxJAG}
D.~A.~Cox. \textit{The homogeneous coordinate ring of a toric variety.} J.~Algebraic Geom. 4, 1995, 17--50.

\bibitem{Cox}
D.~A.~Cox, J.~B.~Little, H.~K.~Schenck. \textit{Toric varieties.} Graduate Studies in Mathematics, 124. American Mathematical Society, Providence, RI, 2011.

\bibitem{De}
M. Demazure, \textit{Sous-groupes alg\'ebriques de rang maximum du groupe de Cremona},  Ann. Sci. \'Ecole Norm. Sup. (4) 3, 1970, 507--588.

\bibitem{Duncan1}
A.~Duncan, \textit{Twisted forms of toric varieties}, Transform.~Groups 21(3), 2016, 763--802.

\bibitem{F93}
W. Fulton, \textit{Introduction to toric varieties,} {Annals of Mathematics Studies, volume 131. Princeton University Press, Princeton, NJ, 1993}

\bibitem{MO} H.~Matsumura, F.~Oort, \textit{Representability of group functors, and automorphisms of algebraic schemes}, Invent.~Math.~4, 1967, 1--25.

\bibitem{MSS}
J.~P.~Moreno, C.~Sancho de Salas, M.~T.~Sancho de Salas, \textit{Automorphism group of a toric variety}, Preprint, 2018. arXiv:1809.09070 [math.AG].

\bibitem{T88} {T. Oda,} \textit{Convex bodies and algebraic geometry. An introduction to the theory of toric varieties},
  {Ergebnisse der Mathematik und ihrer Grenzgebiete
    (3), volume 15. Springer-Verlag, Berlin, 1988.}

\bibitem{Sansuc}
{J.-J.~Sansuc,} \textit{Groupe de Brauer et arithm\'etique des groupes alg\'ebriques lin\'eaires sur un corps de nombres,} J.~Reine Angew.~Math.~327, 1981, 12--80.

\bibitem{SGA3}
\textit{Sch\'emas en groupes}. S\'eminaire de G\'eom\'etrie Alg\'ebrique du Bois Marie 1962/64. Dirig\'e par M.~Demazure et A.~Grothendieck.
Lecture Notes in Mathematics, 151, 152, 153. Springer-Verlag, Berlin, 1970; Documents Math\'ematiques, 7, 8, Soci\'et\'e Math\'ematique de France, Paris, 2011.

\end{thebibliography}
\end{document}